\documentclass[a4paper,12pt]{article}
\usepackage{amsthm}
\usepackage{amsmath}
\usepackage{amssymb,authblk}
\usepackage{t1enc}
\usepackage[utf8]{inputenc}
\usepackage{amsrefs}
\usepackage{float}
\usepackage{tikz}
\usetikzlibrary{calc}

\newtheorem{thm}{Theorem}[section]
\newtheorem{lemma}[thm]{Lemma}
\newtheorem*{lemma*}{Lemma}

\newtheorem{defi}[thm]{Definition}
\newtheorem{conj}[thm]{Conjecture}
\newtheorem{claim}[thm]{Claim}

\theoremstyle{remark}
\newtheorem*{remark}{Remark}

\title{Properties of minimally $t$-tough graphs} 
\author[1]{Gyula Y. Katona\thanks{kiskat@cs.bme.hu}}
\author[2]{Dániel Soltész\thanks{solteszd@math.bme.hu}} 
\author[3]{Kitti Varga\thanks{vkitti@cs.bme.hu}}
\affil[1,2,3]{Department of Computer Science and
Information Theory,
  Budapest University of Technology and Economics}
	\affil[1]{MTA-ELTE Numerical Analysis and Large Networks Research Group}
\date{\today}

\begin{document}
\maketitle

\begin{abstract}
A graph $G$ is minimally $t$-tough if the toughness of $G$ is $t$ and the deletion of any edge from $G$ decreases the toughness. Kriesell conjectured that for every minimally $1$-tough graph the minimum degree $\delta(G)=2$. We show that in every minimally $1$-tough graph $\delta(G) \le \frac{n}{3} + 1$. We also prove that every minimally $1$-tough, claw-free graph is a cycle. On the other hand, we show that for every positive rational number $t$ any graph can be embedded as an induced subgraph into a minimally $t$-tough graph. 
\end{abstract}
\section{Introduction}

All graphs considered in this paper are finite, simple and undirected. Let $d(v)$ denote the degree of a vertex $v$, $\omega(G)$ denote the number of components, $\alpha(G)$ denote the independence number and $\delta(G)$ denote the minimum degree of a graph $G$.

\begin{defi}
 A graph $G$ is $k$-connected, if it has at least $k+1$ vertices and remains connected whenever fewer than $k$ vertices are removed. The connectivity of $G$, denoted by $\kappa(G)$, is the largest k for which $G$ is $k$-connected.
\end{defi}

The more edges a graph has, the larger its connectivity can be, so the graphs, which are $k$-connected and have the fewest edges for this property, may be interesting.

\begin{defi}
  A graph $G$ is minimally $k$-connected, if $\kappa(G) = k$ and ${\kappa(G - e)} < k$ for all $e \in E(G)$.
\end{defi}

Clearly, all degrees of a $k$-connected graph have to be at least $k$. On the other hand, Mader proved that the minimum degree of every minimally $k$-connected graph is exactly $k$.

\begin{thm}[Mader \cite{ende}]
 Every minimally $k$-connected graph has a vertex of degree $k$.
 \label{mader}
\end{thm}

The notion of toughness was introduced by Chv\'{a}tal \cite{chvatal} in 1973.

\begin{defi}
  Let $t$ be a positive real number. A graph $G$ is called $t$-tough,
  if ${\omega(G - S) \le |S|/t}$ for any cutset $S$ of $G$. 
  The toughness of $G$, denoted by $\tau(G)$, is the largest $t$ for
  which G is $t$-tough, taking $\tau(K_n) = \infty$ for all $n \ge 1$.
  
  We say that a cutset $S \subseteq V(G)$ is a tough set if $\omega(G - S) = |S|/\tau(G)$.
\end{defi}


We can define an analogue of minimally $k$-connected graphs for the notion of toughness.

\begin{defi}
  A graph $G$ is said to be minimally $t$-tough, if $\tau(G) = t$ and
  $\tau(G - e) < t$ for all $e \in E(G)$.
\end{defi}

It follows directly from the definition that every $t$-tough graph is $2t$-connected, implying $\kappa (G) \ge 2 \tau (G)$ for noncomplete graphs. Therefore, the minimum degree of any 1-tough graph is at least 2. Kriesell conjectured that the analogue of Mader's theorem holds for minimally $1$-tough graphs. 

\begin{conj}[Kriesell \cite{kriesell}]     
 \label{kriesell}
 Every minimally $1$-tough graph has a vertex of degree $2$.
\end{conj}

A 1-tough graph is always 2-connected, however, a minimally 1-tough graph is not necessarily minimally 2-con\-nected (see Figure~\ref{min2_petersen}), so Mader's theorem cannot be applied.

\begin{figure}[H]
\centering
\begin{tikzpicture}
 \tikzstyle{vertex}=[draw,circle,fill=black,minimum size=4,inner sep=0]
 
 \node[vertex] (a1) at (90:1.5) {};
 \node[vertex] (a2) at (90+72:1.5) {};
 \node[vertex] (a3) at (90+2*72:1.5) {};
 \node[vertex] (a4) at (90+3*72:1.5) {};
 \node[vertex] (a5) at (90+4*72:1.5) {};
 
 \node[vertex] (b1) at (90:0.75) {};
 \node[vertex] (b2) at (90+72:0.75) {};
 \node[vertex] (b3) at (90+2*72:0.75) {};
 \node[vertex] (b4) at (90+3*72:0.75) {};
 \node[vertex] (b5) at (90+4*72:0.75) {};
 
 \draw (a1) -- (a2);
 \draw (a2) -- (a3);
 \draw (a3) -- (a4);
 \draw (a4) -- (a5);
 \draw (a5) -- (a1);
 
 \draw (b1) -- (b3);
 \draw (b2) -- (b4);
 \draw (b3) -- (b5);
 \draw (b4) -- (b1);
 \draw (b5) -- (b2);
 
 \draw (a1) -- (b1) node[pos=0.5, right] {\footnotesize $\hspace{-2pt} e$};
 \draw (a2) -- (b2);
 \draw (a5) -- (b5);
 
 \node at (-1.5,1.5) {$G$};
\end{tikzpicture}
\caption{A minimally 1-tough but not minimally 2-connected graph. The graph $G-e$ is still $2$-connected.}
\label{min2_petersen}
\end{figure}
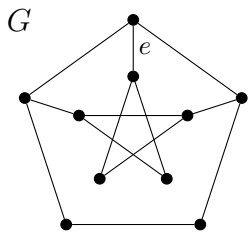

A natural approach to Kriesell's conjecture is to prove upper bounds on $\delta(G)$ for minimally $1$-tough graphs. Kriesell's conjecture states that ${\delta(G) \le 2}$, and the best known upper bound follows easily from Dirac's theorem, yielding $\delta(G) \le n/2$. Our main result is an improvement on the current upper bound by a constant factor.


\begin{thm}
 \label{fotetel}
 Every minimally 1-tough graph has a vertex of degree at most $\frac{n}{3} + 1$.
\end{thm}

Toughness is related to the existence of Hamiltonian cycles. If a graph contains a Hamiltonian cycle, then it is necessarily 1-tough. The converse is not true, a well-known counterexample is the Petersen graph. It is easy to see that every minimally 1-tough, Hamiltonian graph is a cycle, since after deleting an edge that is not contained by the Hamiltonian cycle, the resulting graph is still $1$-tough.

Let us introduce a class of graphs that is frequently studied while dealing with problems related to Hamiltonian cycles. 






\begin{defi}
  The graph $K_{1,3}$ is called a claw. A graph is said to be claw-free, if it does not contain a claw as an induced subgraph.
\end{defi}

Problems about connectivity in claw-free graphs can be handled more easily, since every vertex of a cutset is adjacent to at most two components. We give a complete characterization of minimally 1-tough, claw-free graphs. 

\begin{thm}
 \label{clawfree-sajat}
 If $G$ is a minimally 1-tough, claw-free graph of order $n \ge 4$, then $G=C_n$.
\end{thm}

Thus we see that Kriesell's conjecture is true in a very strong sense if the graph is claw-free. Or equivalently, the family of minimally 1-tough, claw-free graphs is small. On the other hand, we show that in general the class of minimally 1-tough graphs is large. 

\begin{thm} \label{beagyazas}
For every positive rational number $t$, any graph can be embedded as an induced subgraph into a minimally $t$-tough graph.
\end{thm}

The paper is organized as follows. In Section \ref{section_main} we prove Theorem \ref{fotetel} which is the main result of this paper. In Section \ref{section_clawfree} we prove Theorem \ref{clawfree-sajat} and in Section \ref{section_beagyazas} we prove Theorem \ref{beagyazas}. 


\section{Proof of the main result} \label{section_main}

Here we prove that every minimally 1-tough graph has a vertex of degree at most $\frac{n}{3} + 1$. First we need a claim that has a key role in the proofs, then we continue with two lemmas.

\begin{claim} \label{min1toughlemma}
 If $G$ is a minimally $1$-tough graph, then for every edge $e \in E(G)$ there exists a vertex set ${S=S(e) \subseteq V(G)}$  with 
 \[ {\omega(G-S) = |S|} \quad \text{and} \quad \omega \big( (G-e)-S \big) = |S| + 1 \text{.} \]
\end{claim}
\begin{proof} 
 Let $e$ be an arbitrary edge of $G$. Since $G$ is minimally 1-tough, ${\tau (G-e) < 1}$, so there exists a cutset ${S=S(e) \subseteq V(G-e)=V(G)}$ in $G-e$ satisfying that $\omega \big( (G-e)-S \big) > |S|$. On the other hand, $\tau (G)=1$, so $\omega (G-S) \le |S|$. This is only possible if $e$ connects two components of $(G-e)-S$, which means $\omega \big( (G-e)-S \big) = |S| + 1$ and $\omega (G-S) = |S|$.
\end{proof}

\begin{defi}
Let $G$ be a minimally 1-tough graph, and $e$ an arbitrary edge of $G$. Let us define $k(e)$ to be the minimal size of the vertex set $S$ guaranteed by Claim \ref{min1toughlemma}.
\end{defi}

In the proof of the next Lemma, we need the following theorem.

\begin{thm}[\cite{shortcut}] \label{shortcut}
Let $G$ be a $2$-connected graph on $n$ vertices with $\delta(G) \ge \big( n+\kappa(G) \big)/3$. Then $G$ is Hamiltonian.
\end{thm}

\begin{lemma} \label{lemma1}
Let $G$ be a minimally 1-tough graph on $n$ vertices with $\delta(G) > \frac{n}{3} + 1$. Then $k(e) > \frac{n}{3}$ for any $e \in E(G)$.
\end{lemma}
\begin{proof}
Let $e$ be an arbitrary edge of $G$. By Claim \ref{min1toughlemma} there exists a number $k=k(e)$ and a set of $k$ vertices, whose removal from $G-e$ leaves exactly $k+1$ connected components. Clearly, there is no edge between two different components except $e$. Among these components there must be one with at most $\left\lfloor \frac{n-k}{k+1} \right\rfloor$ vertices, and inside this component every vertex can have at most $\left\lfloor \frac{n-k}{k+1} \right\rfloor -1$ neighbors. If this component has size 1, then the vertex inside it has degree at most $k+1$ in $G$, so
\[ \frac{n}{3} + 1 < \delta(G) \le k+1 \text{,} \]
which means that $k > \frac{n}{3}$. Otherwise there exists a vertex in this component, which is not an endpoint of $e$, so its degree in $G$ is at most
\[ \left\lfloor \frac{n-k}{k+1} \right\rfloor -1 + k \le \frac{n-k}{k+1} -1 + k = \frac{n+k^2-k-1}{k+1} \text{.} \]

Consider the function
\[f_n (k) = \frac{n+k^2-k-1}{k+1} \text{.} \]
Note that for any fixed $n$, $f$ is monotone decreasing in $k$ if $0 \le k \le \sqrt{n+1} - 1$ and monotone increasing if $\sqrt{n+1} - 1 < k \le n-1$.

We show that if $k \le \frac{n}{3}$, then $\delta(G) \le \frac{n}{3} + 1$.

\textit{Case 1:} $2 \le k \le \frac{n}{3}$. Since $f_n (k)$ is an upper bound of the minimum degree, it is enough to show that $f_n(k) \le \frac{n}{3} + 1$. The above mentioned property of the function implies that it is enough to show this for $k=2$ and $k=\frac{n}{3}$.
\[ f_n (2) = \frac{n+1}{3} < \frac{n}{3} + 1 \text{,} \]
\[ f_n\left(\frac{n}3\right) = \frac{n^2 + 6n - 9}{3n + 9} = \frac{(n+3)^2-18}{3(n+3)} < \frac{n+3}{3} = \frac{n}{3} + 1 \text{.} \]

\textit{Case 2:} $k=1$. Since $G$ is $1$-tough, $\kappa(G) \ge 2$. Let $e = uv$ be such an edge, for which $k(e) = 1$. Then there exists a single vertex $w$ that disconnects the graph $G-e$, so $\{ u,w \}$ or $\{ v,w \}$ is a cutset in $G$. Thus $\kappa(G) = 2$. Since $\delta(G) > \frac{n}{3} + 1$, by Theorem \ref{shortcut} $G$ is Hamiltonian, but $G \ne C_n$, which contradicts the fact that $G$ is minimally 1-tough.
\end{proof}

Let us define the {\em open neighborhood} an edge $f=\{a,b\}$. It is the set of vertices adjacent to either $a$ or $b$ excluding $a$ and $b$.
\begin{lemma} \label{work}
If $G$ is a minimally 1-tough graph with $\delta(G) > \frac{n}{3}+1$, then there are two vertices $a,b \in V(G)$ connected by an edge $f \in E(G)$ such that their open neighborhood has size more than $\frac{2n}{3}-1$.
\end{lemma}
\begin{proof}
Lemma \ref{lemma1} implies that $k(e) > \frac{n}{3}$ for all $e \in E(G)$. Let us fix an arbitrary edge $e\in E(G)$, and we define $x := k-\frac{n}{3}$. It is easy to see that $0 < x < \frac{n}{6}$, because removing at least $\frac{n}{2}$ vertices does not leave enough components. Let $B := S(e)$ denote the set of the removed vertices and let $A$ denote the set of the remaining vertices. Then $|A| = \frac{2n}{3}-x$, $|B|= \frac{n}{3} + x$ and by the choice of $B$ the number of components in $G-B$ is also $\frac{n}{3} + x$.

Our strategy is to prove that there exists a vertex $b \in B$ having at least $\frac{n}{3} +1$ neighbors in $A$ and among these neighbors there exists a vertex $a$ contained by a component of size at most 2 after the removal of $B$, see Figure~\ref{strategy}. Since $a$ has more than $\frac{n}{3}-1$ neighbors in $B \setminus \{b\}$ and $b$ has at least $\frac{n}{3}$ neighbors in $A \setminus \{ a \}$, their open neighborhood has size more than
\[ \frac{n}{3} - 1 + \frac{n}{3} = \frac{2n}{3}-1 \text{.} \]

\begin{figure}[H]
\begin{center}
\begin{tikzpicture}[scale=2]
\tikzstyle{vertex}=[draw,circle,fill=black,minimum size=3,inner sep=0]

\draw[rounded corners=5] (-1.5,-0.25) rectangle (1.5,0.25);
\node[vertex] (b) at (-1.25,0) [label= {[shift={(0.175,-0.6)}] $b$ }] {} ; 
\node[vertex] at (-0.75,0) {};
\node[vertex] at (-0.25,0) {};
\node[vertex] at (0.25,0) {};
\node[vertex] at (0.75,0) {};
\node[vertex] at (1.25,0) {};

\draw (-2.5,1.5) circle (0.20);
\draw (-1.9,1.5) circle (0.20);
\draw (-1.3,1.5) circle (0.20);
\draw (-0.3,1.5) ellipse (0.55 and 0.45);
\draw (1,1.5) ellipse (0.55 and 0.45);
\draw (2.3,1.5) ellipse (0.55 and 0.45);

\node[vertex] (a) at (-1.975,1.5) [label={[label distance=-3]45: $a$}] {};

\draw (a) -- (b) node[pos=0.5, left] {$f$};

\draw (a) -- (-1.4,0.9);
\draw (a) -- (-1.1,0.9);
\draw (a) -- (-0.8,0.9);
\draw (a) -- (-0.5,0.9);

\draw (b) -- (-0.1,0.6);
\draw (b) -- (-0.5,0.6);
\draw (b) -- (-0.9,0.6);
\draw (b) -- (-1.3,0.6);

\node[vertex] at (0.75,1.5) {};
\node[vertex] at (1.25,1.5) {};
\draw (0.75,1.5) -- (1.25,1.5) node[pos=0.5, above] {$e$};
\end{tikzpicture}
\caption{Finding an edge $f$ for which $G-f$ is $1$-tough.}
\label{strategy}
\end{center}
\end{figure}
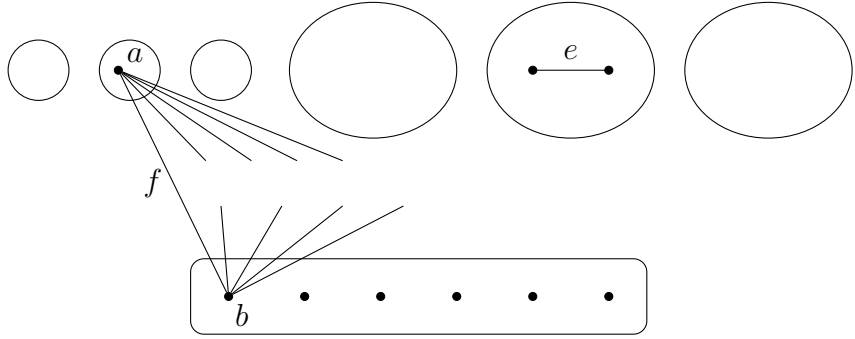

Suppose to the contrary that there exist no such vertices $a$ and $b$. Let $e(A,B)$ denote the number of edges between $A$ and $B$. We give a lower and an upper bound for $e(A,B)$, then we show that the lower bound is greater than the upper bound, which leads us to a contradiction.

\begin{enumerate}
 \item[I.] Lower bound: $e(A,B) > \frac{n^2}{9} + \frac{n}{3} + nx + x - 4 x^2$.

It is well-known, that the number of the edges in a graph with $n_0$ vertices and $k_0$ components is at most $\binom{n_0-k_0+1}{2}$. Hence the number of the edges in $A$ is at most 
\[ \binom{\left(\frac{2n}{3}-x\right) - \left(\frac{n}{3} + x\right) + 1}{2} = \binom{\frac{n}{3} - 2x + 1}{2} \text{.} \]
Since every degree is more than $\frac{n}{3} + 1$, the following lower bound can be given for $e(A,B)$.
\begin{gather*}
 \left( \frac{2n}{3} - x \right) \left( \frac{n}{3} + 1 \right) - 2 \cdot \binom{\frac{n}{3} - 2x + 1}{2} = \\
 =\left( \frac{2n}{3} - x \right) \left( \frac{n}{3} + 1 \right) - \left( \frac{n}{3} - 2x + 1 \right) \left( \frac{n}{3} - 2x \right) = \\
 = \frac{n^2}{9} + \frac{n}{3} + nx + x - 4 x^2
\end{gather*}

\item[II.] Upper bound: $e(A,B) < \frac{n}{3} \left( \frac{n}{3} + 1 \right) + x \left( \frac{n}{2} - 3x \right)$.

To prove the upper bound, we need the following claim. 
\begin{claim}
 After the removal of $B$ there are at least $\left( \frac{n}{6} + 2x \right)$ components of size at most $2$.
\end{claim}
\begin{proof}
After the removal of $B$ the remaining graph has $\frac{2n}{3} - x$ vertices and $\frac{n}{3} + x$ components. In every component there must be at least one vertex, so the other $\left( \frac{2n}{3} - x \right) - \left( \frac{n}{3} + x \right)$ vertices can create at most 
\[ \frac{1}{2} \left( \left( \frac{2n}{3} - x \right) - \left( \frac{n}{3} + x \right) \right) = \frac{1}{2} \cdot \left( \frac{n}{3} - 2x \right) = \frac{n}{6} - x \]
components with size at least 3. So there must be at least
\[ \left( \frac{n}{3} + x \right) - \left( \frac{n}{6} - x \right) = \frac{n}{6} + 2x \]
components having size at most 2.
\end{proof}

Now we return to the proof of the upper bound. After removing $B$, the components of size at most 2 have more than $\frac{n}{3}$ neighbors in $B$. By our assumption, each of these neighbors is connected to less than $\frac{n}{3} + 1$ vertices in $A$. Then all the remaining less than $x$ vertices in $B$ are such that their neighbors in $A$ lie in a component of size at least $3$. So all these remaining less than $x$ vertices in $B$ can be adjacent to at most
\[ \left( \frac{2n}{3} - x \right) - \left( \frac{n}{6} + 2x \right) = \frac{n}{2} - 3x \]
vertices in $A$.

Hence, there are more than $\frac{n}{3}$ vertices in $B$ that have less than $\frac{n}{3}+1$ neighbors in $A$ and the remaining less than $x$ vertices in $B$ have at most $\frac{n}{2} - 3x$ neighbors in $A$, see Figure~\ref{upper_bound}.

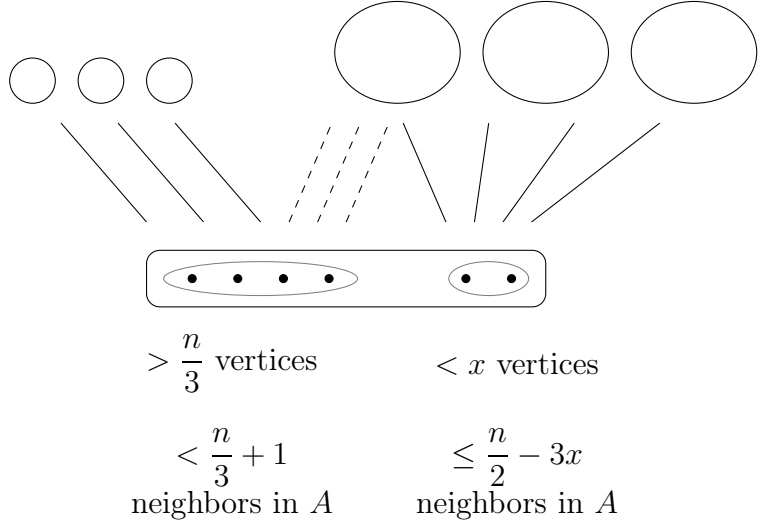
\begin{figure}[H]
\begin{center}
\begin{tikzpicture}[scale=1.5]
\tikzstyle{vertex}=[draw,circle,fill=black,minimum size=3,inner sep=0]

\draw[rounded corners=5] (-1.5,-0.75) rectangle (2,-0.25);
\node[vertex] at (-1.1,-0.5) {};
\node[vertex] at (-0.7,-0.5) {};
\node[vertex] at (-0.3,-0.5) {};
\node[vertex] at (0.1,-0.5) {};
\node[vertex] at (1.3,-0.5) {};
\node[vertex] at (1.7,-0.5) {};

\draw (-2.5,1.25) circle (0.20);
\draw (-1.9,1.25) circle (0.20);
\draw (-1.3,1.25) circle (0.20);
\draw (0.7,1.5) ellipse (0.55 and 0.45);
\draw (2,1.5) ellipse (0.55 and 0.45);
\draw (3.3,1.5) ellipse (0.55 and 0.45);
\draw[opacity=0.5] (-0.5,-0.5) ellipse (0.85 and 0.15);
\draw[opacity=0.5] (1.5,-0.5) ellipse (0.35 and 0.15);

\draw (-1.5,0) -- (-2.25,0.875);
\draw (-1,0) -- (-1.75,0.875);
\draw (-0.5,0) -- (-1.25,0.875);
\draw (1.125,0) -- (0.75,0.875);
\draw (1.375,0) -- (1.5,0.875);
\draw (1.625,0) -- (2.25,0.875);
\draw (1.875,0) -- (3,0.875);
q\draw[dashed] (-0.25,0) -- (0.125,0.875);
\draw[dashed] (0,0) -- (0.375,0.875);
\draw[dashed] (0.25,0) -- (0.625,0.875);

\node at (-0.75,-1.25) {$\displaystyle > \frac{n}{3}$ vertices};
\node at (-0.75,-2.25) {\begin{tabular}{c} $\displaystyle < \frac{n}{3} + 1$ \\ neighbors in $A$ \end{tabular}};
\node at (1.75,-1.25) {$< x$ vertices};
\node at (1.75,-2.25) {\begin{tabular}{c} $\displaystyle \le \frac{n}{2} - 3x$ \\ neighbors in $A$ \end{tabular}};
\end{tikzpicture}
\caption{Giving an upper bound for $e(A,B)$.}
\label{upper_bound}
\end{center}
\end{figure}

Now we show that $\frac{n}{2} - 3x > \frac{n}{3} + 1$. Intuitively this means that $e(A,B)$ is maximal if the components of size at most 2 have as few neighbors as possible. This is an easy corollary of the following claim.

\begin{claim} \label{average}
 For the vertices of $B$, the average number of neighbors in $A$ is more than $\frac{n}{3} + 1$.
\end{claim}
\begin{proof}
 It is already proved that the number of the edges between $A$ and $B$ is more than
 \begin{gather*}
  \frac{n^2}{9} + \frac{n}{3} + nx + x - 4 x^2 \text{.}
 \end{gather*}
 We need to show that
 \[ \frac{n^2}{9} + \frac{n}{3} + nx + x - 4 x^2 > |B|\left( \frac{n}{3} + 1 \right) = \left( \frac{n}{3} + x \right) \left( \frac{n}{3} + 1 \right) \text{.} \]
 Transforming it into equivalent forms, we can see that this inequality holds.
 \begin{align*}
  \frac{n^2}{9} + \frac{n}{3} + nx + x - 4 x^2 & > \frac{n^2}{9} +  \frac{n}{3} + \frac{n}{3} x + x \\
  \frac{2n}{3} x & > 4 x^2 \\
  \frac{n}{6} & > x
 \end{align*}
\end{proof}

If $\frac{n}{2} - 3x > \frac{n}{3}+1$ did not hold, then each vertex in $B$ could be adjacent to at most $\frac{n}{3}+1$ vertices in $A$, which contradicts Claim \ref{average}. So the number of the edges between $A$ and $B$ is less than
\[ \frac{n}{3} \cdot \left( \frac{n}{3} + 1 \right) + x \left( \frac{n}{2} - 3x \right) \text{,} \]
thus the proof of the upper bound is complete. 

\end{enumerate}
 
Clearly, the lower bound cannot be greater than the upper bound, so
\begin{align*}
 \frac{n^2}{9} + \frac{n}{3} + nx + x - 4 x^2 & < \frac{n}{3} \cdot \left( \frac{n}{3} + 1 \right) + x \left( \frac{n}{2} - 3x \right) \text{,} \\
 0 & < x^2 - \left( \frac{n}{2} + 1 \right) x \text{,} \\
 0 & < x \left[ x - \left( \frac{n}{2} + 1 \right) \right] \text{.}
\end{align*}

This contradicts the fact that $0 < x < \frac{n}{6}$.
\end{proof}

\textbf{Proof of Theorem \ref{fotetel}.}  Suppose to the contrary that
$\delta(G) > \frac{n}{3}+1$ and consider the edge
$f=ab$ guaranteed by Lemma \ref{work}. By Claim
\ref{min1toughlemma} there exist $k > \frac{n}{3}$ vertices, whose
removal from $G-f$ leaves $k+1$ connected components, see
Figure~\ref{main_idea}.

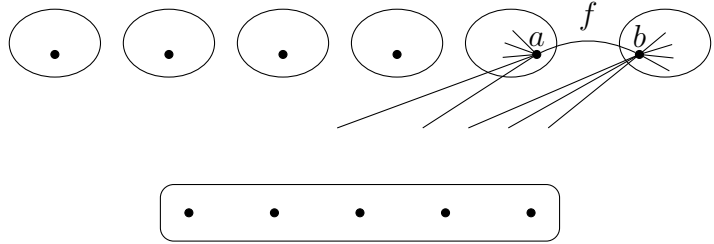
\begin{figure}[H]
\begin{center}
\begin{tikzpicture}[scale=1.5]
\tikzstyle{vertex}=[draw,circle,fill=black,minimum size=3,inner sep=0]

\draw[rounded corners=5] (-1.75,-0.25) rectangle (1.75,0.25);
\node[vertex] at (-1.5,0) {};
\node[vertex] at (-0.75,0) {};
\node[vertex] at (0,0) {};
\node[vertex] at (0.75,0) {};
\node[vertex] at (1.5,0) {};

\draw (-2.675,1.5) ellipse (0.40 and 0.3);
\draw (-1.675,1.5) ellipse (0.40 and 0.3);
\draw (-0.675,1.5) ellipse (0.40 and 0.3);
\draw (0.325,1.5) ellipse (0.40 and 0.3);
\draw (1.325,1.5) ellipse (0.40 and 0.3);
\draw (2.675,1.5) ellipse (0.40 and 0.3);

\node[vertex] at (-2.675,1.4) {};
\node[vertex] at (-1.675,1.4) {};
\node[vertex] at (-0.675,1.4) {};
\node[vertex] at (0.325,1.4) {};

\node[vertex] (a) at (1.55,1.4) [label={[label distance=-2.5]90:$a$}] {};
\node[vertex] (b) at (2.45,1.4) [label={[label distance=-3]90:$b$}] {};

\draw (a) to [bend left=25] node[pos=0.5, above] {$f$} (b);

\begin{scope}[shift={(a)}]
 \draw (0,0) -- (135:0.3);
 \draw (0,0) -- (160:0.3);
 \draw (0,0) -- (185:0.3);
 
 \draw (0,0) -- (-1.75,-0.65);
 \draw (0,0) -- (-1,-0.65);
\end{scope}

\begin{scope}[shift={(b)}]
 \draw (0,0) -- (40:0.3);
 \draw (0,0) -- (17:0.3);
 \draw (0,0) -- (-6:0.3);
 \draw (0,0) -- (-29:0.3);
 
 \draw (0,0) -- (-1.5,-0.65);
 \draw (0,0) -- (-1.15,-0.65);
 \draw (0,0) -- (-0.8,-0.65);
\end{scope}

\end{tikzpicture}
\caption{There are too many neighbors of $a$ and $b$.}
\label{main_idea}
\end{center}
\end{figure}

For this we need $k+1 >  \frac{n}{3} + 1$ independent vertices (one in each of  the $k+1$ components), two of them are $a$ and $b$, and the rest of them cannot be adjacent either to $a$ or to $b$. However, there are less than
\[ n - \left( \frac{2n}{3} - 1 \right) = \frac{n}{3} + 1 < k+1 \] such
vertices, since $a$ and $b$ have more than $\frac{2n}{3}-1$ different
neighbors. So $G-f$ is 1-tough, which is a contradiction. \hfill\qed

\section{Claw-free graphs} \label{section_clawfree}

In this section we prove that minimally 1-tough, claw-free graphs are just cycles (of length at least 4). By the following theorem, the toughness of claw-free graphs can be easily computed.

\begin{thm}[\cite{claw}] \label{clawfree}
 If $G$ is a noncomplete claw-free graph, then $2\tau(G) = \kappa(G)$.
\end{thm}

In our proof we need the following lemmas.

\begin{lemma} \label{clawfreelemma}
 Let $G$ be a claw-free graph with $\tau(G) = t$ and $S$ a tough set. Now the vertices of $S$ have neighbors in exactly two components of $G-S$, and the components of $G-S$ have exactly $2t$ neighbors (in $S$).
\end{lemma}

Lemma \ref{clawfreelemma} follows from the proof of Theorem \ref{clawfree}, which we do not present here, it can be found as Theorem $10$ in \cite{claw}.

\begin{lemma} \label{no2degtriangle}
 If $G$ is a minimally $1$-tough graph, then every vertex of any triangle has degree at least 3.
\end{lemma}

\begin{proof}
Suppose to the contrary that $\{u,v,w\}$ is a triangle and $u$ has degree 2. Let $e$ be the edge connecting $v$ and $w$. By Claim \ref{min1toughlemma} there exists a vertex set $S$ such that $\omega(G-S)=|S|$ and $\omega \big( (G-e)-S \big)=|S|+1$. Clearly, $u \in S$ and $v,w \not\in S$. Since the neighbors of $u$ are adjacent, and $G$ is 1-tough
\[ |S| = \omega(G-S) = \omega \big( G - (S \setminus \{ u \}) \big) \le |S| - 1 \text{,} \]
which is a contradiction.
\end{proof}

\bigskip

\textbf{Proof of Theorem \ref{clawfree-sajat}.}
 Suppose to the contrary that $G$ has a vertex of degree at least $3$. Since $G$ is claw-free, some neighbors of this vertex must be connected, hence there must be a triangle in $G$. Let us denote the vertices of this triangle by $\{u,v,w\}$. 
 
 \begin{claim} \label{haromszog_mintoughsetjei}
  For some edge of the triangle $\{u,v,w\}$, the vertex set guaranteed by Claim \ref{min1toughlemma} has size at least two.
 \end{claim} 
 \begin{proof}
 
 Suppose to the contrary that for each edge the corresponding vertex set has size 1. Thus for each edge this set must consist of the third vertex of the triangle, i.e.~for the edges $e_1 = vw$, $e_2 = uw$ and $e_3 = uv $, these sets are $S_1 := S(e_1) = \{ u \}$, $S_2 := S(e_2) = \{ v \}$ and $S_3 := S(e_3) = \{ w \}$.
 
 Let $L_1$ and $L_2$ denote the connected components of $(G-e_1) - S_1$ containing $v$, $w$ respectively. Now the components of $(G-e_2)-S_2$ must be $L_2$ and $\big( L_1 - \{ v \} \big) \cup \{ u \}$, and the components of $(G-e_3)-S_3$ must be $L_1$ and $\big( L_2 - \{ w \} \big) \cup \{ u \}$. So $u$ cannot have any neighbors in $L_2 - \{ w \}$ and in $L_1 - \{ v \}$, see Figure~\ref{egerek}. This means that $u$ has only two neighbors $v$ and $w$, which is a contradiction by Lemma \ref{no2degtriangle}.
 
 \begin{figure}[H]
 \begin{center}
 \begin{tikzpicture}[scale=2]
 \tikzstyle{vertex}=[draw,circle,fill=black,minimum size=3,inner sep=0]
 
 \node[vertex] (u) at (-1,0) [label={[label distance=-3]270:$u$}] {};
 
 \draw (-2,1) ellipse (0.6 and 0.3);
 \draw (0,1) ellipse (0.6 and 0.3);
 
 \node at (-2,1.5) {$L_1$};
 \node at (0,1.5) {$L_2$};
 
 \node[vertex] (v) at (-1.625,1) [label={[label distance=-2.5]180: $v$}] {};
 \node[vertex] (w) at (-0.375,1) [label={[label distance=-3]0:$w$}] {};
 
 \draw (v) to [bend left=25] node[pos=0.5, above] {$e_1$} (w);
 
 \draw (u) -- (w);
 \draw (u) -- (v);
 \draw (u) -- ($(u)!0.6!(0,1)$);
 \draw (u) -- ($(u)!0.6!(0.5,1)$);
 \draw (u) -- ($(u)!0.6!(-2,1)$);
 \draw (u) -- ($(u)!0.6!(-2.5,1)$);
 
 \begin{scope}[shift={(0,-2.5)}]
 \node[vertex] (v) at (-1,0) [label={[label distance=-3]270:$v$}] {};
 
 \draw (-2,1) ellipse (0.6 and 0.3);
 \draw (0,1) ellipse (0.6 and 0.3);
 
 \node at (-2,1.5) {$\big( L_1 - \{v\} \big) \cup \{u\}$};
 \node at (0,1.5) {$L_2$};
 
 \node[vertex] (u) at (-1.625,1) [label={[label distance=-2.5]180: $u$}] {};
 \node[vertex] (w) at (-0.375,1) [label={[label distance=-3]0:$w$}] {};
 
 \draw (u) to [bend left=25] node[pos=0.5, above] {$e_2$} (w);
 
 \draw (u) -- (v);
 \draw (w) -- (v);
 \draw (v) -- ($(v)!0.6!(0,1)$);
 \draw (v) -- ($(v)!0.6!(0.5,1)$);
 \end{scope}
 \end{tikzpicture}
 \caption{The vertex $u$ cannot have any neighbors in $L_2 - \{w\}$.}
 \label{egerek}
 \end{center}
 \end{figure}
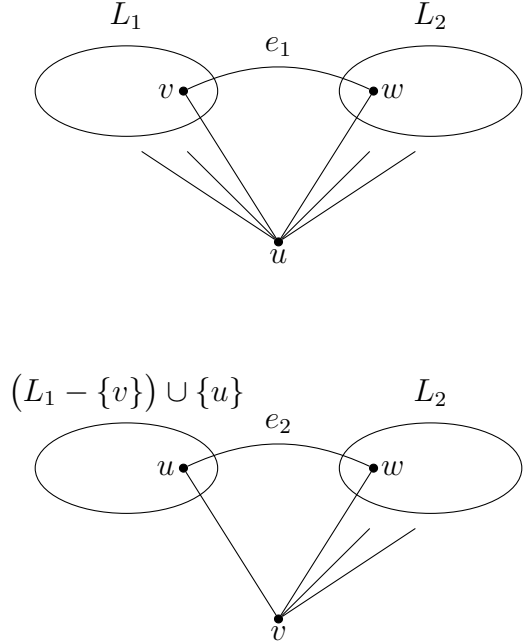
 \end{proof}
 
 By Claim \ref{haromszog_mintoughsetjei} we can assume that for the edge $e = uw$, the vertex set $S = S(e)$ garanteed by Claim \ref{min1toughlemma} has size at least 2. This means that $S$ is a cutset. Since $\omega(G)=|S|$, $S$ is a tough set. So by Lemma \ref{clawfreelemma} the component of $G-S$ that contains the edge $e$ has exactly two neighbors in $S$. One such neighbor must be $v$, and let us denote the other neighbor by $v_2$. Observe that the set $\{v,v_2\}$ is a tough set. Let $L_1, L_2$ denote the connected components of $(G-e)-\{v,v_2\}$ containing $u,w$ respectively and let $L_3$ denote the third connected component, see Figure \ref{3eyedalien}.
 
  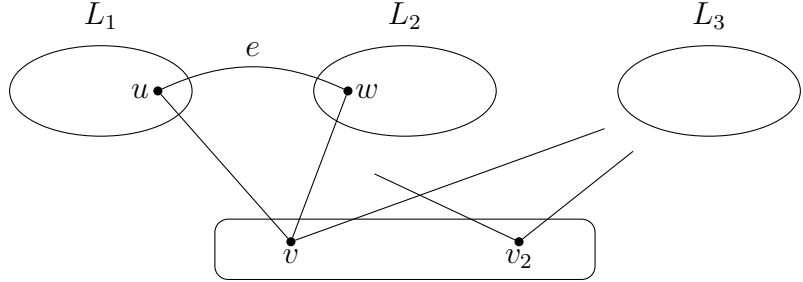
\begin{figure}[H]
 \begin{center}
 \begin{tikzpicture}[scale=2]
 \tikzstyle{vertex}=[draw,circle,fill=black,minimum size=3,inner sep=0]
 
 \draw[rounded corners=5] (-1.25,-0.25) rectangle (1.25,0.15);
 \node[vertex] (v) at (-0.75,0) [label={[label distance=-3]270:$v$}] {};
 \node[vertex] (v2) at (0.75,0) [label={[label distance=-3]270:$v_2$}] {};
 
 \draw (-2,1) ellipse (0.6 and 0.3);
 \draw (0,1) ellipse (0.6 and 0.3);
 \draw (2,1) ellipse (0.6 and 0.3);
 
 \node at (-2,1.5) {$L_1$};
 \node at (0,1.5) {$L_2$};
 \node at (2,1.5) {$L_3$};
 
 \node[vertex] (u) at (-1.625,1) [label={[label distance=-2.5]180: $u$}] {};
 \node[vertex] (w) at (-0.375,1) [label={[label distance=-3]0:$w$}] {};
 
 \draw (u) to [bend left=25] node[pos=0.5, above] {$e$} (w);
 
 \draw (u) -- (v);
 \draw (w) -- (v);
 \draw (v) -- ($(v)!0.75!(2,1)$);
 \draw (v2) -- (-0.2,0.45);
 \draw (v2) -- ($(v2)!0.6!(2,1)$);
 \end{tikzpicture}
 \caption{The tough set $\{ v, v_2 \}$ and the sets $L_1, L_2, L_3$.}
 \label{3eyedalien}
 \end{center}
 \end{figure}
 
 \textit{Case 1:} both $L_1$ and $L_2$ have size at least 2.
 
 Now $\{ v, v_2, u \}$ is a tough set. Using Lemma \ref{clawfreelemma} we can conclude that $v_2$ has no neighbors in $L_2$ (since $v$ and $u$ have neighbors in $L_2$), so $v_2$ must have neighbors in $L_1 - \{ u \}$, see Figure~\ref{L1>1}. Using the same argument for the tough set $\{ v, v_2, w \}$, we can conclude that $v_2$ has neighbors in $L_2 - \{ w \}$. Then there is a claw in the graph (it is formed by $v_2$ and one of its neighbors in each of the components $L_1 - \{ u \}$, $L_2 - \{ w \}$ and $L_3$), which is a contradiction. So we can assume that $L_1 = \{ u \}$.
 
 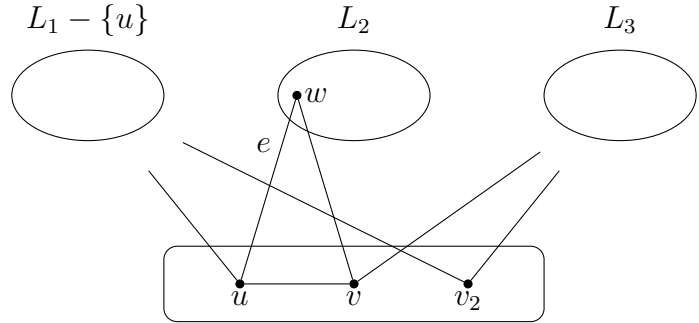
\begin{figure}[H]
 \begin{center}
 \begin{tikzpicture}[scale=2]
 \tikzstyle{vertex}=[draw,circle,fill=black,minimum size=3,inner sep=0]
 
 \draw[rounded corners=5] (-1.25,-0.5) rectangle (1.25,0);
 \node[vertex] (u) at (-0.75,-0.25) [label={[label distance=-3]270:$u$}] {};
 \node[vertex] (v) at (0,-0.25) [label={[label distance=-3]270:$v$}] {};
 \node[vertex] (v2) at (0.75,-0.25) [label={[label distance=-3]270:$v_2$}] {};
 
 \draw (-1.75,1) ellipse (0.5 and 0.3);
 \draw (0,1) ellipse (0.5 and 0.3);
 \draw (1.75,1) ellipse (0.5 and 0.3);
 
 \node at (-1.75,1.5) {$L_1 - \{ u \}$};
 \node at (0,1.5) {$L_2$};
 \node at (1.75,1.5) {$L_3$};
 
 \node[vertex] (w) at (-0.375,1) [label={[label distance=-3]0:$w$}] {};
 
 \draw (u) to node[pos=0.75, left] {$e$} (w);
 
 \draw (u) -- (v);
 \draw (w) -- (v);
 \draw (v2) -- ($(v2)!0.75!(-1.75,1)$);
 \draw (u) -- ($(u)!0.6!(-1.75,1)$);
 \draw (v2) -- ($(v2)!0.6!(1.75,1)$);
 \draw (v) -- ($(v)!0.7!(1.75,1)$);
 \end{tikzpicture}
 \caption{If $|L_1| > 1$.}
 \label{L1>1}
 \end{center}
 \end{figure}
 
 \textit{Case 2:} $L_2$ has size at least 2 (and $L_1 = \{ u \}$).
 
 Now $\{ v, v_2, w \}$ is a tough set, so by Lemma \ref{clawfreelemma} $v_2$ is not adjacent to $u$, so $u$ is a vertex of degree 2, which contradicts Lemma \ref{no2degtriangle}.
 
 \textit{Case 3:} $L_2 = \{ w \}$ (and $L_1 = \{ u \}$).
 
 By Lemma \ref{no2degtriangle}, $N(u) = \{ w, v, v_2 \}$ and $N(w) = \{ u, v, v_2 \}$. Consider the edge $f = uv$ and let $S' := S(f)$ be a vertex set garanteed by Claim~\ref{min1toughlemma}. Clearly, $w \in S'$, and by Lemma \ref{clawfreelemma} $w$ has neighbors in exactly two components of $G-S'$. By the choice of $S'$, 
 the vertices $u$ and $v$ are in the same component in $G-S'$, so $v_2$ must be in a different component, which contradicts the fact that $u$ and $v_2$ are adjacent. \hfill \qed

\section{Embedding graphs into a minimally $t$-tough graph} \label{section_beagyazas}

In this section we show that for every positive rational number $t$, any graph can be embedded as an induced subgraph into a minimally $t$-tough graph. Our proof is constructive. Different constructions are used for $t \geq 1$ and $t < 1$. For this we need a definition and the following well-known exercises from \cite{lovasz}.

\begin{defi}
 A graph $G$ is called $\alpha$-critical, if $\alpha(G-e) > \alpha(G)$ for all $e \in E(G)$.
\end{defi}

\begin{lemma}[Problem $13$ of §8 in \cite{lovasz}] \label{alphacritical}
Every graph can be embedded as an induced subgraph into an $\alpha$-critical graph. 
\end{lemma}

\begin{lemma}[Problem $14$ of §8 in \cite{lovasz}] \label{blowup}
If we replace a vertex of an $\alpha$-critical graph with a clique, and connect every neighbor of the original vertex with every vertex in the clique, then the resulting graph is still $\alpha$-critical. 
\end{lemma}

Now we proceed with the proof of the case $t \geq 1$. 

\begin{thm} \label{t>1}
For every positive rational number $t \ge 1$, any graph $G$ can be embedded as an induced subgraph into a minimally $t$-tough graph.
\end{thm}
\begin{proof}
Let $n$ denote the number of vertices in $G$ and let $a,b \in \mathbb{N}$ be such that $t=a/b$. By Lemma \ref{alphacritical} it is enough to consider the case where $G$ is $\alpha$-critical. By Lemma \ref{blowup} we can also assume that $b$ divides $n$. 

Our strategy is to embed $G$ as an induced subgraph into a graph $H$, which is not necessarily minimally $t$-tough yet, but has the following two properties: $\tau(H) = t$ and deleting any edge of the induced subgraph of $H$ that is isomorphic to $G$, lowers the toughness of $H$. Then if we repeatedly remove an edge from $H$ that does not lower the toughness, the remaining graph will be minimally $t$-tough (since no further edges can be removed), and the edges corresponding to the subgraph $G$ are left intact. We define $H$ as follows. Let $N$ be a large integer to be specified later. Let 
$$V:= \{v_1, \ldots, v_n\} \text{,} $$
$$ W:= \{w_1, \ldots, w_{(t-1)n + \alpha(G)} \} $$
and for each $i\in [n]$ let
$$ U_i:= \{u_{i,1}, \ldots, u_{i,N} \}. $$
Then let
\[ U:= \bigcup_{i=1}^{n} U_i \]
and
$$V(H):= V \cup U \cup W. $$
Place the graph $G$ on the vertices of $V$. For all $i \in [n]$ place a clique on $U_i$ and connect every vertex of $U_i$ to every vertex of $W$ and also to the vertex $v_i$. Note that $\alpha(H)$ does not depend on $N$, so let us choose $N$ such that $N > t \alpha(H)$. Also note that $(t-1)n + \alpha(G)$ is an integer since $b$ divides $n$. See Figure~\ref{random1}.

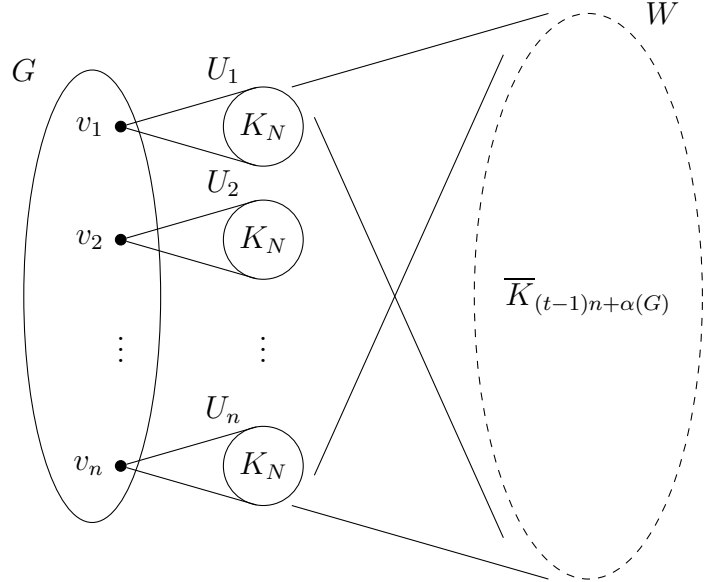
\begin{figure}[H] 
\begin{center}
\begin{tikzpicture}[scale=1.5]
 \tikzstyle{vertex}=[draw,circle,fill=black,minimum size=4,inner sep=0]
 
 \draw (0,0) ellipse (0.6 and 2);
 \node at (-0.6,2) {$G$};
 
 \node[vertex] (v1) at (0.25,1.5) [label=left:{$v_1$}] {};
 \node[vertex] (v2) at (0.25,0.5) [label=left:{$v_2$}] {};
 \node at (0.25,-0.4) {$\vdots$};
 \node[vertex] (vn) at (0.25,-1.5) [label=left:{$v_n$}] {};
 
 \draw (1.5,1.5) ellipse (0.35 and 0.35);
 \draw (1.5,0.5) ellipse (0.35 and 0.35);
 \node at (1.5,-0.4) {$\vdots$};
 \draw (1.5,-1.5) ellipse (0.35 and 0.35);
 
 \node at (1.5,1.5) {$K_N$};
 \node at (1.5,0.5) {$K_N$};
 \node at (1.5,-1.5) {$K_N$};
 
 \node at (1.15,2) {$U_1$};
 \node at (1.15,1) {$U_2$};
 \node at (1.15,-1) {$U_n$};
 
 \draw[dashed] (4.35,0) ellipse (1 and 2.5);
 \node at (5,2.5) {$W$};
 
 \node at (4.35,0) {$\overline{K}_{(t-1)n+\alpha(G)}$};
 
 \begin{scope}[shift={(1.5,1.5)}]
  \draw (v1) -- (100:0.35);
  \draw (v1) -- (260:0.35);
 \end{scope}
 \begin{scope}[shift={(1.5,0.5)}]
  \draw (v2) -- (100:0.35);
  \draw (v2) -- (260:0.35);
 \end{scope}
 \begin{scope}[shift={(1.5,-1.5)}]
  \draw (vn) -- (100:0.35);
  \draw (vn) -- (260:0.35);
 \end{scope}
 
 \draw (1.75,1.85) -- (4,2.5);
 \draw (1.75,-1.85) -- (4,-2.5);
 \draw ($(1.85,1.55)!0.05!(3.8,2.2)$) -- ($(1.85,-1.55)!0.90!(3.8,-2.2)$);
 \draw ($(1.85,1.55)!0.90!(3.8,2.2)$) -- ($(1.85,-1.55)!0.05!(3.8,-2.2)$);
\end{tikzpicture}
\caption{The graph $H$.}
\label{random1}
\end{center}
\end{figure}

\begin{claim}
For each edge $e \in E(G)$, $\tau(H-e) < t$. 
\end{claim} 
\begin{proof}
Since $G$ is $\alpha$-critical, there is an independent set in $H-e$ of size $\alpha(G)+1$  among the vertices of $V$. Let us delete $W$ and all the vertices of $G$ except this independent set. We deleted 
\[ n - \big( \alpha(G)+1 \big) + (t-1)n + \alpha(G)=tn-1 \]
vertices from $H-e$ and the resulting graph has $n$ connected components, thus $H-e$ is not $t$-tough.
\end{proof}

\begin{claim}
$\tau(H) \ge t$.
\end{claim}
\begin{proof}
Suppose to the contrary that there exists a cutset $X \subseteq V(H)$ such that $\omega(H - X) > |X|/t$ and $|X|$ is minimal. For any $i \in [n]$, $U_i \not\subseteq X$, otherwise
\[ \alpha(H) < \frac{N}{t} \le \frac{|X|}{t} < \omega(H - X) \leq \alpha(H) \text{,} \]
which is a contradiction. Since $X$ is minimal, we can assume that for any $i \in [n]$, $U_i \cap X = \emptyset$, since removing only a proper subset of $U_i$ does not disconnect anything from the graph. Thus $W \subseteq X$, otherwise $H - X$ is still connected. Let us denote the number of vertices of $V \cap X$ by $y$. The independence number of the subgraph of $H$ spanned by $V \cup U$ is clearly $n$, thus $\omega(H - X) \leq n$. Which yields 
\[ \frac{y+(t-1)n+ \alpha(G)}{t} = \frac{|X|}{t} < \omega(H - X) \leq n \text{,} \]
so
\[ y+ \alpha(G) < n \text{.} \]

On the other hand, there are at most $\alpha(G)$ components of $\omega(H - X)$ that contain a vertex of $V$, the other connected components are some of the $U_j$, thus $\omega(H - X) \leq y+\alpha(G)$. Which yields 
\[ \frac{y+(t-1)n+ \alpha(G)}{t} = \frac{|X|}{t} < \omega(H - X) \leq y + \alpha(G) \text{,} \]
\[ \frac{(t-1)n}{t} < \left( 1- \frac{1}{t} \right) \big( y + \alpha(G) \big) \text{,} \]
\[ n < y + \alpha(G) \text{,} \]
which is a contradiction.
\end{proof}

\begin{claim}
$\tau(H) \le t$.
\end{claim}
\begin{proof}
Let $I \subseteq V$ be an independent set of size $\alpha(G)$ in the subgraph of $H$ induced by $V$. Let $X:= (V \setminus I) \cup W$. Then 
\[ |X| = n - \alpha(G) + (t-1)n + \alpha(G) = tn \]
and $\omega(H-X) = n$.
\end{proof}

Thus we conclude that $\tau(H) = t$ and the proof is complete.
\end{proof}

\begin{remark}
A different construction can be obtained as follows. Set $N=1$ instead of $N > t \alpha(G)$, and change the size of $W$, but connect every vertex of $W$ to every vertex in $V$, and add a new independent set of vertices $W'$ which is connected only to $W$ as a complete bipartite graph. The fact that we can assume that $G$ is $\alpha$-critical and isolated vertex free, and the appropriate choice of the sizes of $W$ and $W'$ gives an other good construction. For more details see the alternate proof of Theorem $1.1$ in \cite{bms}.
\end{remark}

Note that the construction in Theorem \ref{t>1} cannot be applied for $t<1$ since in general $(t-1)n-\alpha(G)$ could be negative. A more interesting reason why this construction does not work in the case $t<1$ is that it does not reward us with enough components after deleting a vertex. Thus the core idea of the case $t<1$ is that we introduce multiple $U_i$ for each vertex in $V$. 
 
\begin{thm}
For every positive rational number $t < 1$, any graph $G$ can be embedded as an induced subgraph into a minimally $t$-tough graph.
\label{t<1}
\end{thm}
\begin{proof}
 Let $n$ denote the number of vertices in $G$ and let $a,b \in \mathbb{N}$ be such that $t=a/b$. Similarly to Theorem \ref{t>1}, we can assume that $G$ is $\alpha$-critical. We will embed $G$ as an induced subgraph into $H$ such that $\tau(H) = a/b$, and $\tau(H-e) < a/b$ for every edge $e$ of $G$. The construction is similar as in Theorem \ref{t>1}, the same letters denote similar vertices. Let $H$ be defined as follows. Let $N$ be a large integer to be specified later. For each $i \in [n]$ and $j \in [b]$, let
 \begin{gather*}
  V_i := \{ v_{i,1}, v_{i,2}, \ldots, v_{i,a} \}\text{,} \qquad V := \bigcup_{i=1}^n V_i \text{,} \\
  U_{i,j} := \{ u_{i,1}, \ldots, u_{i,N} \} \text{,} \qquad U := \bigcup _{j=1}^{b} \bigcup_{i=1}^n U_{i,j} \text{,} \\
  W := \{ w_1, \ldots, w_{a \cdot \alpha(G)} \} \text{,} \qquad W' := \{ w'_1, \ldots, w'_{(b-1) \cdot \alpha(G)} \} \text{,} \\
  V(H) := V \cup U \cup W \cup W'\text{.}
 \end{gather*}
Place the graph $G$ on the vertices $v_{1,1}, \ldots, v_{n,1}$ and for each $i\in [n]$ place a clique on $V_i$. For each $i \in [n]$,  $j \in [b]$ put a clique on $U_{i,j}$ and connect every vertex in $U_{i,j}$ to every vertex in $W$ and $V_i$. Connect every vertex in $W$ to every vertex in $W'$. Observe that $\alpha(H)$ does not depend on $N$, so let $N > t \alpha(H)$. See Figure~\ref{random2}.

\begin{figure}[H]
\begin{center}
\begin{tikzpicture}[scale=1.5]
 \tikzstyle{vertex}=[draw,circle,fill=black,minimum size=4,inner sep=0]
 
 \draw (-0.75,0) ellipse (0.65 and 3);
 \node at (-1.35,3) {$G$};
 
 \node[vertex] (v1) at (-0.5,2.25) [label=left:{$v_{1,1}$}] {};
 \node[vertex] (v2) at (-0.5,0.75) [label=left:{$v_{2,1}$}] {};
 \node at (-0.5,-0.65) {$\vdots$};
 \node[vertex] (vn) at (-0.5,-2.25) [label=left:{$v_{n,1}$}] {};
 
 \draw (-0.15,2.25) ellipse (0.5 and 0.35);
 \draw (-0.15,0.75) ellipse (0.5 and 0.35);
 \draw (-0.15,-2.25) ellipse (0.5 and 0.35);
 
 \node at (0.1,2.75) {$V_1$};
 \node at (0.1,1.25) {$V_2$};
 \node at (0.1,-1.75) {$V_n$};
 
 \node at (0.1,2.25) {$K_a$};
 \node at (0.1,0.75) {$K_a$};
 \node at (0.1,-2.25) {$K_a$};
 
 \draw (1.5,2.65) ellipse (0.25 and 0.25);
 \node at (1.5,2.32) {\footnotesize{$\vdots$}};
 \draw (1.5,1.85) ellipse (0.25 and 0.25);
 \draw (1.5,1.15) ellipse (0.25 and 0.25);
 \node at (1.5,0.82) {\footnotesize{$\vdots$}};
 \draw (1.5,0.35) ellipse (0.25 and 0.25);
 \node at (1.5,-0.65) {$\vdots$};
 \draw (1.5,-1.85) ellipse (0.25 and 0.25);
 \node at (1.5,-2.18) {\footnotesize{$\vdots$}};
 \draw (1.5,-2.65) ellipse (0.25 and 0.25);
 
 \node at (1.5,2.65) {\footnotesize{$K_N$}};
 \node at (1.5,1.85) {\footnotesize{$K_N$}};
 \node at (1.5,1.15) {\footnotesize{$K_N$}};
 \node at (1.5,0.35) {\footnotesize{$K_N$}};
 \node at (1.5,-1.85) {\footnotesize{$K_N$}};
 \node at (1.5,-2.65) {\footnotesize{$K_N$}};
 
 \node at (1.8,3) {\footnotesize{$U_{1,1}$}};
 \node at (1.8,2.15) {\footnotesize{$U_{1,b}$}};
 \node at (1.8,1.45) {\footnotesize{$U_{2,1}$}};
 \node at (1.8,0.65) {\footnotesize{$U_{2,b}$}};
 \node at (1.8,-1.55) {\footnotesize{$U_{n,1}$}};
 \node at (1.8,-3.05) {\footnotesize{$U_{n,b}$}};
 
 \draw (0.3,2.6) -- (1.3,2.9);
 \draw (0.3,1.9) -- (1.3,1.6);
 \draw ($(0.3,2.6)!0.075!(1.3,1.6)$) -- ($(0.3,2.6)!0.925!(1.3,1.6)$);
 \draw ($(0.3,1.9)!0.075!(1.3,2.9)$) -- ($(0.3,1.9)!0.925!(1.3,2.9)$);
 
 \draw (0.3,1.1) -- (1.3,1.4);
 \draw (0.3,0.4) -- (1.3,0.1);
 \draw ($(0.3,1.1)!0.075!(1.3,0.1)$) -- ($(0.3,1.1)!0.925!(1.3,0.1)$);
 \draw ($(0.3,0.4)!0.075!(1.3,1.4)$) -- ($(0.3,0.4)!0.925!(1.3,1.4)$);
 
 \draw (0.3,-1.9) -- (1.3,-1.6);
 \draw (0.3,-2.6) -- (1.3,-2.9);
 \draw ($(0.3,-1.9)!0.075!(1.3,-2.9)$) -- ($(0.3,-1.9)!0.925!(1.3,-2.9)$);
 \draw ($(0.3,-2.6)!0.075!(1.3,-1.6)$) -- ($(0.3,-2.6)!0.925!(1.3,-1.6)$);
 
 \draw[dashed] (4,0) ellipse (0.75 and 2.5);
 \node at (4.4,2.75) {$W$};
 
 \draw[dashed] (6.75,0) ellipse (0.75 and 2.75);
 \node at (7.25,2.9) {$W'$};
 
 \node at (4,0) {$\overline{K}_{a \cdot \alpha(G)}$};
 \node at (6.75,0) {$\overline{K}_{(b-1) \cdot \alpha(G)}$};
 
 \draw (1.75,2.85) -- (3.65,2.5);
 \draw (1.75,-2.85) -- (3.65,-2.5);
 \draw ($(1.8,2.85)!0.05!(3.6,-2.5)$) -- ($(1.8,2.5)!0.95!(3.6,-2.5)$);
 \draw ($(1.8,-2.85)!0.05!(3.6,2.5)$) -- ($(1.8,-2.5)!0.95!(3.6,2.5)$);
 
 \draw (4.4,2.5) -- (6.35,2.7);
 \draw (4.4,-2.5) -- (6.35,-2.7);
 \draw ($(4.5,2.5)!0.05!(6.25,-2.75)$) -- ($(4.5,2.5)!0.95!(6.25,-2.75)$);
 \draw ($(4.5,-2.5)!0.05!(6.25,2.75)$) -- ($(4.5,-2.5)!0.95!(6.25,2.75)$);
\end{tikzpicture}
\caption{The graph $H$.}
\label{random2}
\end{center}
\end{figure}
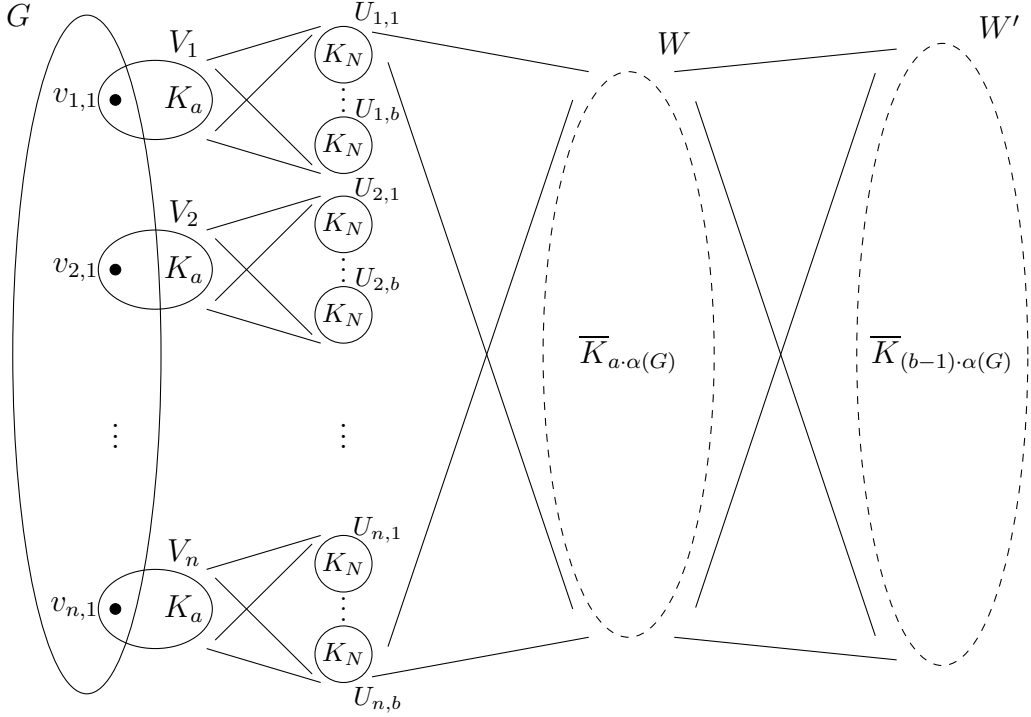

\begin{claim}
For each edge $e \in E(G)$, $\tau(H-e) < a/b$.
\end{claim}
\begin{proof}
By the $\alpha$-criticality of $G$, the graph $G - e$ has an independent set $I \subseteq V(G)$ of size $\alpha(G)+1$. Let us delete the vertices of $W$ and for every $i \in [n]$, if $v_i \notin I$ then let us also delete $V_i$. We deleted exactly 
\[ \big( n-(\alpha(G)+1) \big)a+ \alpha(G) \cdot a= (n-1)a \]
vertices. After the deletion, the vertices of $I$ are in different connected components and the vertices of $W'$ are isolated. There are $\big( n-(\alpha(G)+1) \big)b$ more connected components not containing any vertex from $G$ but containing vertices from $U$. Thus the resulting graph has 
\[ \big( \alpha(G)+1 \big) + (b-1) \cdot \alpha(G) + \big( n-(\alpha(G)+1) \big)b= (n-1)b+1 \]
connected components which is one more than an $a/b$-tough graph could have. 
\end{proof}

\begin{claim}
 $\tau(H) \ge a/b$. 
\end{claim}
\begin{proof}
Suppose to the contrary that there exists a set of vertices $X \subseteq V(H)$ such that $\omega(H - X) > |X|/t \text{.}$
First we show that some convenient assumptions can be made for $X$. 

\begin{lemma} We can assume that $X$ has the following properties.
\begin{enumerate}
 \item[(1)] If some vertices have the same closed (or open) neighborhood, then either all or none of them are in $X$.
 \item[(2)] For every $i \in [n]$ and $j \in [b]$, $U_{i,j} \cap X = \emptyset$.
 \item[(3)] $W \subseteq X$.
 \item[(4)] For all $i \in [n]$, either $V_i \subseteq X$ or $V_i \cap X = \emptyset$.
 \item[(5)] $W' \cap X = \emptyset$.
\end{enumerate}
\end{lemma}
\begin{proof} $\empty$
\begin{enumerate}
 \item[(1)] Removing only a proper subset of such a vertex set does not disconnect anything from the graph.
 \item[(2)] Since for all $i \in [n]$ and $j \in [b]$ the closed neighborhoods of the vertices of $U_{i,j}$ are identical, by (1) either $U_{i,j} \subseteq X$ or $U_{i,j} \cap X = \emptyset$. But by the choice of $N$, $U_{i,j} \not\subseteq X$, otherwise
 \[ \alpha(H) \leq \frac{N}{t} \leq \frac{|X|}{t} < \omega(H-X) \leq \alpha(H) \text{,} \]
 which is a contradiction. 
 \item[(3)] Otherwise by (2) $H-X$ would be connected. 
 \item[(4)] For all $i \in [n]$, the closed neighborhood of $v_{i,1}$ is larger than the closed neighborhood of the vertices $v_{i,2}, \ldots, v_{i,n}$, so if $v_{i,1} \not\in X$, then we can assume that $V_i \cap X = \emptyset$. If $v_{i,1} \in X$, then we can assume that $V_i \subseteq X$, since adding the rest of $V_i$ to $X$ would increase the size of $X$ by at most $a-1$ and the number of the components by exactly $b-1$. This preserves the property that $\omega(H-X) > |X|/t$, as $a < b$.
 \item[(5)] It is a trivial consequence of (3). 
\end{enumerate}
\end{proof}

Let $y$ be the number of the sets $V_i$ that are subsets of $X$. There are at most $\alpha(G)$ components of $H-X$ that contain a vertex from $V$, the other connected components are some of the components of $U$ and every vertex of $W'$. Thus $\omega(H-X) \leq \alpha(G) + by +(b-1) \cdot \alpha(G).$ By our assumption that $|X|/t < \omega(H - X)$ this yields 
\[ \frac{ay+a\alpha(G)}{t}=\frac{|X|}{t}<\omega(H - X) \leq \alpha(G) + by +(b-1) \cdot \alpha(G) \text{.} \]
Since $t=a/b$,
\[ b \big( ay+a\alpha(G) \big)< a \big( by+b\alpha(G) \big) \text{,} \]
which is a contradiction.
\end{proof}

\begin{claim}
 $\tau(H) \le a/b$.
\end{claim}
\begin{proof}
Let $I \subseteq V(G) = \{ v_{1,1}, \ldots, v_{n,1} \}$ be any independent set of size $\alpha(G)$ and consider the set 
\[ X = \left( \bigcup \{ V_i \mid v_{1,i} \notin I \} \right) \cup W \text{.} \]
Since 
\[ |X| = a \big( n-\alpha(G) \big) + a\alpha(G) = an \]
and $H-X$ has exactly
\[ \alpha(G)+ b\big(n-\alpha(G)\big)+(b-1) \cdot \alpha(G) = bn \]
connected components.
\end{proof}

So $\tau (H) =a/b$ and the proof is complete. 
\end{proof}

\section*{Acknowledgement}
We thank the anonymous reviewers for their careful reading of our manuscript and their many insightful comments and suggestions. All authors are partially supported by the grant of the National Research, Development and Innovation Office – NKFIH, No. 108947. The first author is also supported by the National Research, Development and Innovation Office – NKFIH, No. 116769. The second author is also supported by the National Research, Development and Innovation Office – NKFIH, No. 120706.

\end{document}